\documentclass[12pt, a4paper]{amsart}
\usepackage{graphicx}

\usepackage[a4paper, top=3cm, left=3cm, right=2cm, bottom=3cm]{geometry}
\usepackage{amsmath, amsthm}
\usepackage{amssymb}
\usepackage{multicol,comment}
\usepackage{amsfonts}
\usepackage{color, enumerate, fancyhdr, dsfont}
\usepackage[utf8]{inputenc}

\usepackage[normalem]{ulem}

    \newcommand{\R}{\mathbb{R}}

    \newcommand{\ran}{\mbox{\rm ran}}

    \newcommand{\seq}{\mbox{\rm sq}}

    \newcommand{\Awf}{\mathcal{A}}
    
    \newcommand{\Cwf}{\mathcal{C}}

    \newcommand{\Iwf}{\mathcal{I}}

    \newcommand{\Mwf}{\mathcal{M}}
    \newcommand{\Nwf}{\mathcal{N}}
    
    \newcommand{\Pwf}{\mathcal{P}}

    \newcommand{\minadd}{\mathrm{minadd}}
    
    \newcommand{\supcof}{\mathrm{supcof}}
     \newcommand{\Seq}{\mathrm{sq}}
   \newcommand{\hgt}{\mathrm{ht}}
    \newcommand{\afrak}{\mathfrak{a}}
    \newcommand{\bfrak}{\mathfrak{b}}
    \newcommand{\cfrak}{\mathfrak{c}}
    \newcommand{\dfrak}{\mathfrak{d}}

    \newcommand{\ifrak}{\mathfrak{i}}

    \newcommand{\ufrak}{\mathfrak{u}}

    \newcommand{\menos}{\smallsetminus}

    \newcommand{\add}{\mbox{\rm add}}
    \newcommand{\cov}{\mbox{\rm cov}}
    \newcommand{\non}{\mbox{\rm non}}
    \newcommand{\cof}{\mbox{\rm cof}}
    
    \newcommand{\limdir}{\mbox{\rm limdir}}

    \newcommand{\Eor}{\mathbb{E}}

    \newcommand{\Por}{\mathbb{P}}
    \newcommand{\Qor}{\mathbb{Q}}
    
    \newcommand{\Sor}{\mathbb{S}}
    \newcommand{\Tor}{\mathbb{T}}

    \newcommand{\cf}{\mbox{\rm cf}}

\title[The covering number of the strong measure zero ideal]{The covering number of the strong measure zero ideal can be above almost everything else}

\author{Miguel A. Cardona}
\address{Institute of Discrete Mathematics and Geometry, TU Wien, Wiedner Hauptstrasse 8--10/104 A--1040 Wien, Austria.}
\email{miguel.montoya@tuwien.ac.at}
\urladdr{https://www.researchgate.net/profile/Miguel\_Cadona\_Montoya}

\author{Diego A. Mej\'ia}
\address{Creative Science Course (Mathematics), Faculty of Science, Shizuoka University, Ohya 836, Suruga-ku, Shizuoka-shi, Japan 422-8529.}
\email{diego.mejia@shizuoka.ac.jp}
\urladdr{http://www.researchgate.com/profile/Diego\_Mejia2}

\author{Ismael E. Rivera-Madrid}
\address{Faculty of Engineering, Instituci\'on Universitaria Pascual Bravo. Calle 73 No. 73A - 226, Medell\'in, Colombia.}
\email{ismael.rivera@pascualbravo.edu.co}

\thanks{This work was supported by the Austrian Science Fund (FWF) P30666 (first author), the Grant-in-Aid for Early Career Scientists 18K13448, Japan Society for the Promotion of Science (second author), and by the grant no. IN201711, Direcci\'on Operativa de Investigaci\'on, Instituci\'on Universitaria Pascual Bravo (second and third authors).}

\subjclass[2010]{03E17, 03E35, 03E40.}

\keywords{Strong measure zero sets, cardinal invariants, Sacks model.}

\begin{document}

\makeatletter
\def\@roman#1{\romannumeral #1}
\makeatother

\theoremstyle{plain}
  \newtheorem{theorem}{Theorem}[section]
  \newtheorem{corollary}[theorem]{Corollary}
  \newtheorem{lemma}[theorem]{Lemma}
  \newtheorem{mainlemma}[theorem]{Main Lemma}
  \newtheorem{prop}[theorem]{Proposition}
  \newtheorem{claim}[theorem]{Claim}
  \newtheorem{exer}[theorem]{Exercise}
\theoremstyle{definition}
  \newtheorem{definition}[theorem]{Definition}
  \newtheorem{example}[theorem]{Example}
  \newtheorem{remark}[theorem]{Remark}
  \newtheorem{context}[theorem]{Context}
  \newtheorem{question}[theorem]{Question}
  \newtheorem{problem}[theorem]{Problem}
  \newtheorem{notation}[theorem]{Notation}

  \newcommand{\azul}[1]{{\color{blue}#1}}
\newcommand{\rojo}[1]{{\color{red}#1}}
\newcommand{\tachar}[1]{{\color{red}\sout{#1}}}
\definecolor{amber}{rgb}{1.0,0.49,0.0}

\definecolor{ogreen}{RGB}{107,142,35}

\newcommand{\verde}[1]{{\color{ogreen}#1}}
\newcommand{\amber}[1]{{\color{amber}#1}}

\newcommand{\Fn}{\mathrm{Fn}}
\newcommand{\leqT}{\preceq_{\mathrm{T}}}
\newcommand{\eqT}{\cong_{\mathrm{T}}}
\newcommand{\la}{\langle}
\newcommand{\ra}{\rangle}
\newcommand{\id}{\mathrm{id}}
\newcommand{\Lv}{\mathrm{Lv}}
\newcommand{\sig}{\boldsymbol{\Sigma}}
\newcommand{\spl}{\mathrm{spl}}
\newcommand{\st}{\mathrm{st}}
\newcommand{\suc}{\mathrm{succ}}
\newcommand{\cosig}{\boldsymbol{\Pi}}
\newcommand{\Lb}{\mathrm{Lb}}
\newcommand{\pw}{\mathrm{pw}}

\newcommand{\SNcal}{\mathcal{SN}}
\newcommand{\Fr}{\mathrm{Fr}}
\newcommand{\Dbf}{\mathbf{D}}
\newcommand{\Cbf}{\mathbf{C}}
\newcommand{\Rbf}{\mathbf{R}}

\newcommand{\Ibb}{\mathbb{I}}
\newcommand{\PTbb}{\mathbb{PT}}
\newcommand{\Qbb}{\mathbb{Q}}
\newcommand{\Tbb}{\mathbb{T}}

\begin{abstract}
   We show that certain type of tree forcings, including Sacks forcing, increases the covering of the strong measure zero ideal $\mathcal{SN}$. As a consequence, in Sacks model, such covering number is equal to the size of the continuum, which indicates that this covering number is consistently larger than any other classical cardinal invariant of the continuum. Even more, Sacks forcing can be used to force that $\mathrm{non}(\mathcal{SN})<\mathrm{cov}(\mathcal{SN})<\mathrm{cof}(\mathcal{SN})$, which is the first consistency result where more than two cardinal invariants associated with $\mathcal{SN}$ are pairwise different. Another consequence is that $\mathcal{SN}\subseteq s^0$ in ZFC where $s^0$ denotes the Marczewski's ideal.
\end{abstract}

\maketitle

\section{Introduction}\label{SecIntro}

This paper is focused on new consistency results about cardinal invariants associated with the strong measure zero ideal. Recall that \emph{the cardinal invariants associated with an ideal $\Iwf\subseteq\Pwf(X)$} are
\begin{itemize}
\item[{}]$\add(\Iwf):=\min\{|\Awf|:\Awf\subseteq\Iwf\text{\ and } \bigcup\Awf\notin\Iwf\}$ \emph{the additivity of $\Iwf$}; 

\item[{}]$\cov(\Iwf):=\min\{|\Cwf|:\Cwf\subseteq\Iwf\text{\ and }\bigcup\Cwf=X\}$ \emph{the covering of $\Iwf$}; 
    
\item[{}]$\non(\Iwf):=\min\{|Z|:Z\subseteq X\text{\ and }Z\notin\Iwf\}$ \emph{the uniformity of $\Iwf$};  
    
\item[{}]$\cof(\Iwf):=\min\{|\Cwf|:\Cwf\subseteq\Iwf\text{\ is cofinal in }\la\Iwf,\subseteq\ra\}$ \emph{the cofinality of $\Iwf$}.  
\end{itemize}

In this context, we assume that ideals on $\Pwf(X)$ contain all the finite subsets of $X$. Under this assumption, the inequalities indicated in Figure \ref{FigId} can be proved in ZFC.

A very classical instance of cardinal invariants is \emph{Cicho\'n's diagram} (Figure \ref{FigCichon}), which is composed by the cardinal invariants associated with the ideal $\Mwf$ of meager subsets of $\R$ and with the ideal $\Nwf$ of Lebesgue-measure subsets of $\R$, by the bounding number $\bfrak$ and dominating number $\dfrak$ (reviewed in Section \ref{SecPre}), and by $\cfrak=|\R|=2^{\aleph_0}$. This diagram is complete in the sense that no other inequalities can be proved (see e.g. \cite{BJ} for all the details).

Denote by $\SNcal$ the ideal of strong measure zero subsets of $\R$. In relation with the cardinal invariants in Cicho\'n's diagram, the following is provable in ZFC:
\begin{enumerate}[({SN}1)]
    \item $\add(\Nwf)\leq\add(\SNcal)$ (Carlson \cite{Carlson}),
    \item $\cov(\Nwf)\leq\cov(\SNcal)\leq\cfrak$,
    \item $\cov(\Mwf)\leq\non(\SNcal)\leq\non(\Nwf)$ and $\add(\Mwf)=\min\{\bfrak,\non(\SNcal)\}$ (Miller \cite{Miller}),
    \item $\cof(\SNcal)\leq 2^\dfrak$ (see \cite{Osuga}).
\end{enumerate}

On the other hand, the following inequalities are \emph{consistent with ZFC}:
\begin{enumerate}[({C}1)]
    \item $\cof(\Mwf)<\add(\SNcal)$ (Goldstern, Judah and Shelah \cite{GJS}),
    \item $\cov(\SNcal)<\add(\Mwf)$ (Pawlikowski \cite{P90}),
    \item $\non(\SNcal)<\min\{\bfrak,\cov(\Nwf)\}$ (Hechler's model followed by random model),
    \item $\cfrak<\cof(\SNcal)$ (from CH),
    \item $\cof(\SNcal)<\cfrak$ (Yorioka \cite{Yorioka}).
\end{enumerate}

\begin{figure}
    \centering
  \includegraphics{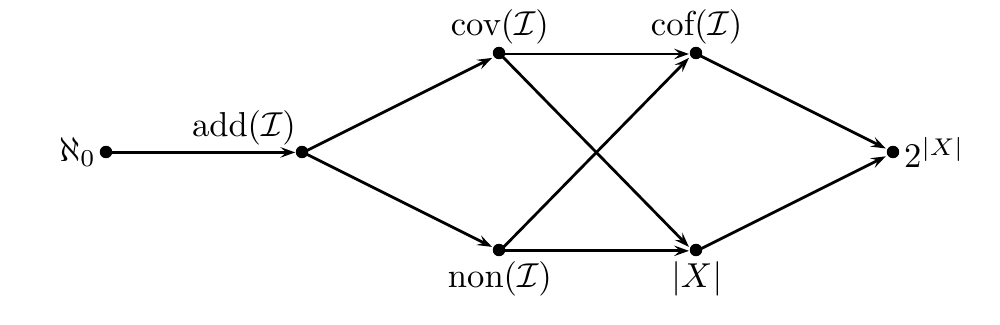}
  \caption{The arrows mean that $\leq$ is provable in ZFC.}
  \label{FigId}
\end{figure}

In fact, the forcing model constructed in (C1) satisfies $\dfrak=\aleph_1$ and $\SNcal=[\R]^{\leq\aleph_1}$, so $\cof(\SNcal)=\cfrak=\aleph_2$ whenever $2^{\aleph_1}=\aleph_2$ in the ground model. The equality $\cof(\SNcal)=\cfrak$ also follows from Borel's conjecture $\SNcal=[\R]^{\leq\aleph_0}$, which was proven consistent with ZFC by Laver \cite{Laver}.

The consistency results above show that no other inequality between $\add(\SNcal)$ and another cardinal in Cicho\'n's diagram can be proved, and the same can be said about $\non(\SNcal)$. However, unsolved problems about $\cov(\SNcal)$ and $\cof(\SNcal)$ still remain.

   \begin{enumerate}[({Q}1)]
       \item Is there a classical cardinal invariant of the continuum (different from $\cfrak$ and $\cof(\SNcal)$) that is an upper bound of $\cov(\SNcal)$? In particular, is $\cov(\SNcal)\leq\cof(\Nwf)$?
       \item Is there a classical cardinal invariant of the continuum (different from $\cov(\Nwf)$ and $\cov(\Mwf)$)\footnote{Obvious lower bounds of $\cof(\SNcal)$ are $\cov(\Nwf)$ and $\cov(\Mwf)$ because of (SN2) and (SN3), respectively.} that is a lower bound of $\cof(\SNcal)$? In particular, is $\add(\Mwf)\leq\cof(\SNcal)$? Is $\cof(\Nwf)\leq\cof(\SNcal)$?
   \end{enumerate}

In this work, we answer (Q1) in the negative, that is, we show that $\cov(\SNcal)$ is consistently larger than $\cof(\Nwf)$ and even larger than any classical cardinal invariant of the continuum (like the \emph{almost disjointedness number} $\afrak$, the \emph{independence number} $\ifrak$ and the \emph{ultrafilter number} $\ufrak$, which are maximal among classical cardinal invariants of the continuum that could be below $\cfrak$). In fact, we show that $\cov(\SNcal)=\cfrak=\aleph_2$ in Sacks model (where any classical cardinal invariant of the continuum is $\aleph_1$).

In addition to this, we prove the consistency of $\non(\SNcal)<\cov(\SNcal)<\cof(\SNcal)$. This is the first consistency result where more than two cardinal invariants associated with $\SNcal$ are pairwise different. For this proof, we use Yorioka's characterization of $\cof(\SNcal)$ (\cite{Yorioka}, Theorem \ref{Ycof} in this text).

The core of these results are Main Lemma \ref{mainlemma} and Theorem \ref{SacksSN}, which states that a type of tree forcings (Definition \ref{Defbtree}), and their iterations, increases $\cov(\SNcal)$. In terms of ideals, this implies that $\SNcal\subseteq s^0$ where  $s^0=\{X\subseteq2^\omega:\forall p\in\Sor\exists q\leq p([q]\cap X=\emptyset)\}$ is the \emph{Marczewski's ideal} (originally defined in \cite{Mar}) and $\Sor$ denotes \emph{Sacks forcing}, so $\cov(s^0)\leq\cov(\SNcal)$.\footnote{Also $\non(\SNcal)\leq\non(s^0)$, but  $\non(s^0)=\cfrak$ because $[2^\omega]^{<\cfrak}\subseteq s^0$.} 

This paper is structured as follows. In Section \ref{SecPre} we review the basic notation and the results this paper is based on. In Section \ref{SecPres} we present preservation results related to the dominating number of $\kappa^\kappa$ for $\kappa$ regular, this to ensure that $\cof(\SNcal)$ can be manipulated as desired via Yorioka's characterization theorem. In Section \ref{SecMain} we prove our main results, even more, we show that a type of tree forcings, when iterated, increases $\cov(\SNcal)$. Section \ref{SecDisc} is dedicated to discussions and open questions.

\begin{figure}
\begin{center}
  \includegraphics{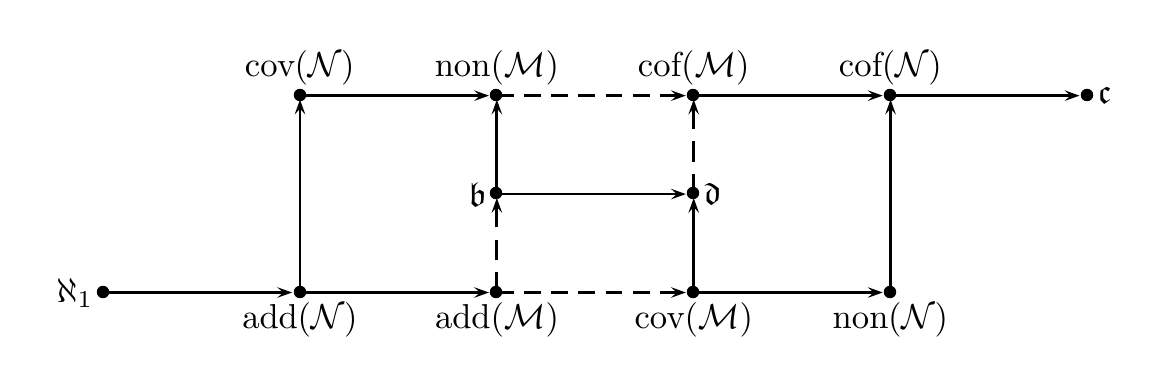}
  \caption{Cicho\'n's diagram. The arrows mean $\leq$ and dotted arrows represent
  $\add(\Mwf)=\min\{\bfrak,\cov(\Mwf)\}$ and $\cof(\Mwf)=\max\{\dfrak,\non(\Mwf)\}$.}
  \label{FigCichon}
\end{center}
\end{figure}



\section{Preliminaries}\label{SecPre}

We start with a short review of the \emph{Tukey order}. A \emph{relational system} is a triplet $\Rbf=\la X,Y,\sqsubset\ra$ where $\sqsubset$ is a relation. Such a relational system has two cardinal invariants associated with it:
\begin{itemize}
    \item[{}] $\bfrak(\Rbf):=\min\{|F|:F\subseteq X\text{\ and }\neg\exists y\in Y\forall x\in X(x\sqsubset y)\}$,
    \item[{}] $\dfrak(\Rbf):=\min\{|D|:D\subseteq Y\text{\ and }\forall x\in X \exists y\in D(x\sqsubset y)\}$.
\end{itemize}
These cardinals do not exists in general, for example, $\bfrak(\Rbf)$ does not exists iff $\dfrak(\Rbf)=1$; likewise, $\dfrak(\Rbf)$ does not exists iff $\bfrak(\Rbf)=1$.

Let $\Rbf':=\la X',Y',\sqsubset'\ra$ be another relational system. Say that \emph{$\Rbf$ is Tukey below $\Rbf'$}, denoted by $\Rbf\leqT\Rbf'$, if there are maps $F:X\to X'$ and $G:Y'\to Y$ such that, for any $x\in X$ and $y'\in Y'$, if $F(x)\sqsubset' y'$ then $x\sqsubset G(y')$. Say that \emph{$\Rbf$ and $\Rbf'$ are Tukey equivalent}, denoted by $\Rbf\eqT\Rbf'$, if $\Rbf\leqT\Rbf'$ and $\Rbf'\leqT\Rbf$. Note that $\Rbf\leqT\Rbf'$ implies $\bfrak(\Rbf')\leq\bfrak(\Rbf)$ and $\dfrak(\Rbf)\leq\dfrak(\Rbf')$.

Let $\kappa$ be an infinite cardinal. For any $x,y\in\kappa^\kappa$ write $x\leq y$ for $\forall i<\kappa(x(i)\leq y(i))$, and define $x<y$ similarly. Define the relation $x\leq^* y$ by $\exists i_0<\kappa\forall i\geq i_0(x(i)\leq y(i))$, which we read \emph{$y$ dominates $x$}. Say that $D\subseteq\kappa^\kappa$ is \emph{dominating (in $\kappa^\kappa$)} if it is cofinal in $\la\kappa^\kappa,\leq^*\ra$, that is, every function in $\kappa^\kappa$ is dominated by some member of $D$. Denote $\Dbf_\kappa:=\la\kappa^\kappa,\kappa^\kappa,\leq^*\ra$ and $\Dbf:=\Dbf_\omega$. Define the cardinal invariants $\bfrak_\kappa:=\bfrak(\Dbf_\kappa)$ and $\dfrak_\kappa:=\dfrak(\Dbf_\kappa)$.
The classical unbounded and dominating numbers are $\bfrak:=\bfrak_\omega$ and $\dfrak:=\dfrak_\omega$, respectively.

The relational system $\Dbf_1$ defined below is relevant in the proof of the main results.

\begin{definition}\label{DefIntRel}
   Denote by $\Ibb$ the set of interval partitions of $\omega$. Define the relational systems $\Dbf_1:=\la\Ibb,\Ibb,\sqsubseteq\ra$ and $\Dbf_2:=\la\Ibb,\Ibb,\ntriangleright\ra$ where, for any $I,J\in\Ibb$,
   \[ I \sqsubseteq J \text{\ iff } \forall^\infty n\exists m(I_m\subseteq J_n); \quad
      I \ntriangleright J \text{\ iff } \forall^\infty n\forall m( I_n\nsupseteq J_m).
   \]
   For each $I\in\Ibb$ we define $f_I:\omega\to\omega$ and $I^{*2}\in\Ibb$ such that $f(n):=\min I_{n}$ and $I^{*2}_n:=I_{2n}\cup I_{2n+1}$. For each increasing $f\in\omega^\omega$ define the increasing function $f^*:\omega\to\omega$ such that $f^*(0)=0$ and $f^*(n+1)=f(f^*(n)+1)$, and define $I^f\in\Ibb$ such that $I_n^f:=[f^*(n),f^*(n+1))$.
\end{definition}

In Blass \cite{blass} it is proved that $\Dbf\eqT\Dbf_1$. For completeness, we present the proof and include $\Dbf_2$.

\begin{lemma}\label{dompart}
    $\Dbf\eqT\Dbf_1\eqT\Dbf_2$. Even more, if $D\subseteq\omega^\omega$ is a dominating family of increasing functions, then $\{I^f:f\in D\}$ is $\Dbf_1$-dominating.
\end{lemma}
\begin{proof}
    To see $\Dbf_1\leqT\Dbf$ note that, for any $I\in\Ibb$ and $f\in\omega^\omega$, if $f\in\omega^\omega$ is increasing then $f_I\leq^* f$ implies $I\sqsubseteq I^f$. Indeed, for $n$ large enough, put $m:=f^*(n)$, so $f^*(n)=m\leq f_I(m)<f_I(m+1)\leq f(m+1)=f^*(n+1)$, that is, $I_m\subseteq I^f_n$.

    For $\Dbf_2\leqT\Dbf_1$ note that, for any $I,J\in\Ibb$, $I\sqsubseteq J$ implies $I\ntriangleright J^{*2}$. Finally, $\Dbf\leqT\Dbf_2$ because, for any increasing $f\in\omega^\omega$ and $I\in\Ibb$, $I^f\ntriangleright I$ implies $f\leq^* f_I$. To show this, notice that $I^f\ntriangleright I$ is equivalent to say that $(f_I(n),f_I(n+1))\cap\ran f^*\neq\emptyset$ for all but finitely many $n$. Split into cases: if $f=\id_\omega$, then $f^*=\id_\omega$, so $(f_I(n),f_I(n+1))\neq\emptyset$ for $n$ large enough. Hence, while $f(n+1)-f(n)=1$, eventually $f_I(n+1)-f_I(n)\geq 2$, which guarantees $f\leq^* f_I$.

    For the second case, assume $f(m_0)>m_0$ for some $m_0<\omega$.\footnote{Since $f$ is increasing, $f\geq\id_\omega$.} This implies that $f(n)>n$ for every $n\geq m_0$. To guarantee $f\leq^* f_I$, it is enough to show that $|I_n\cap \ran f|\geq 2$ for infinitely many $n$ (recall that $I_n\cap\ran f\neq\emptyset$ for large enough $n$). If $n\in\omega$ is  large enough, then there is some $m<\omega$ such that $f_I(n)<f^*(m)<f_I(n+1)$. On the other hand, since $(f_I(n+1),f_I(n+2))\cap\ran f^*\neq\emptyset$, $f^*(m+1),f(f^*(m))\in I_n\cup I_{n+1}$. This clearly implies that either $I_n$ or $I_{n+1}$ intersects $\ran f$ in 2 or more points.
\end{proof}

Fix $b:\omega\to \omega\menos\{0\}$.
Denote $\prod b:=\prod_{i<\omega}b(i)$ and $\Seq_{<\omega}(b):=\bigcup_{n<\omega}\prod_{i<n}b(i)$. For each $s\in\Seq_{<\omega}(b)$ define $[s]:=[s]_b:=\{x\in\prod b: s\subseteq x\}$.
The space $\prod b$, as a topological space endowed with the product topology where each $b(i)$ has the discrete topology, is a compact Polish space with open basis $\{[s]:s\in\Seq_{<\omega}(b)\}$. Even more, $\prod b$ is a perfect space whenever $b\nleq^*1$.

For combinatorial purposes, we use the notion of strong measure zero in $\prod b$.

\begin{definition}
For each $\sigma \in (\Seq_{<\omega}(b))^\omega$ define $\hgt_{\sigma}\in\omega^{\omega}$ by $\hgt_{\sigma}(i):=|\sigma(i)|$.

Say that $X\subseteq \prod b$ \textit{has strong measure zero} iff for every $f\in\omega^\omega$ there is some $\sigma\in (\Seq_{<\omega}(b))^\omega$ with  $\hgt_\sigma=f$ such that $X\subseteq\bigcup_{i<\omega}[\sigma(i)]$.

Denote $\SNcal(\prod b):=\{X\subseteq\prod b :X\text{\ has strong measure zero}\}$. Likewise, we use the notation $\SNcal(\R)$ and $\SNcal([0,1])$.
\end{definition}

When $b\nleq^* 1$, the map $F_b:\prod b\to[0,1]$ defined by
\[F_b(x):=\sum_{n<\omega}\frac{x(n)}{\prod_{i\leq n}b(i)}\]
is a continuous onto function, which is one-to-one on the set of sequences in $\prod b$ that are not eventually constant. This map preserves sets between $\SNcal(\prod b)$ and $\SNcal([0,1])$ via images and pre-images, therefore, the value of the cardinal invariants associated with $\SNcal$ do not depend on the space $\prod b$, neither on $\R$ or $[0,1]$.

For $\sigma \in (\Seq_{<\omega}(b))^\omega$ define
\[[\sigma]_\infty:=[\sigma]_{b,\infty}=\{x \in \prod b:\exists^{\infty}{n < \omega}^{}(\sigma(n) \subseteq x)\}=\bigcap_{n<\omega} \bigcup_{m \geqslant n}[\sigma(m)].\]

The following characterization of $\SNcal$ is quite practical in terms of the sets above.

\begin{lemma}\label{charSN}
    Let $X\subseteq\prod b$ and let $D\subseteq\omega^\omega$ be a dominating family. Then $X\in\SNcal(\prod b)$ iff for every $f\in D$ there is some $\sigma\in (\Seq_{<\omega}(b))^\omega$ with $\hgt_\sigma=f$ such that $X\subseteq[\sigma]_\infty$.
\end{lemma}

Now we focus on $b=2$ (as a constant function). Denote $\pw_k:\omega\to\omega$ the function defined by $\pw_k(i):=i^k$, and define the relation $\ll$ on $\omega^\omega$ as follows:
\[f\ll g \text{\ iff } \forall{k<\omega}(f\circ\pw_k\leq^* g).\]



\begin{definition}[Yorioka {\cite{Yorioka}}]\label{DefYorio} For each $f \in \omega^\omega$ define
    \[\mathcal{I}_{f}:=\{X\subseteq 2^{\omega}:\exists{\sigma \in (2^{<\omega})^{\omega}}(X \subseteq [\sigma]_\infty\text{\ and }\hgt_{\sigma}\gg f )\}.\]
 Any family of the form $\mathcal{I}_{f}$ with $f$ increasing is called a \textit{Yorioka ideal}.
\end{definition}

Yorioka \cite{Yorioka} has proved that $\Iwf_f$ is a $\sigma$-ideal when $f$ is increasing. By Lemma \ref{charSN} it is clear that $\SNcal=\bigcap\{\Iwf_f:f\textrm{\ increasing}\}$. Denote
\[\minadd:=\min\{\add(\Iwf_f):f\text{\ increasing}\},\quad \supcof:=\sup\{\cof(\Iwf_f):f\text{\ increasing}\}.\]

It is known that $\add(\Nwf)\leq\minadd\leq\add(\Mwf)$ and $\cof(\Mwf)\leq\supcof\leq\cof(\Nwf)$ (see \cite{Osuga,CM}), even more, it is not hard to see that $\minadd\leq\add(\SNcal)$. Yorioka's characterization of $\cof(\SNcal)$ is established as follows.

\begin{theorem}[Yorioka {\cite{Yorioka}}]\label{Ycof}
   If $\minadd=\supcof=\kappa$ then $\cof(\SNcal)=\dfrak_\kappa$.\footnote{In Yorioka's original result it is further assumed that $\dfrak=\cov(\Mwf)=\kappa$, but this is now known to be redundant.}
\end{theorem}


To finish this section, we review how to increase $\dfrak_\kappa$ with $\kappa$-Cohen reals. Denote by $\Fn_{<\kappa}(I,J)$ the poset of partial functions from $I$ into $J$ with domain of size $<\kappa$, ordered by $\supseteq$.

\begin{lemma}\label{kappaCohen}
    Let $\kappa$ and $\lambda$ be infinite cardinals. If $\lambda>\kappa^{<\kappa}$ then $\Fn_{<\kappa}(\lambda\times\kappa,\kappa)$ forces $\dfrak_\kappa\geq\lambda$.
\end{lemma}
\begin{proof}
   Let $\gamma<\lambda$ and let $\{\dot{y}_\alpha:\alpha<\gamma\}$ be a set of $\Fn_{<\kappa}(\lambda\times\kappa,\kappa)$-names of functions in $\kappa^\kappa$. Since this poset is $(\kappa^{<\kappa})^+$-cc, there is some $S\in[\lambda]^{<\lambda}$ such that each $\dot{y}_\alpha$ is a $\Fn_{<\kappa}(S\times\kappa,\kappa)$-name. A genericity argument guarantees that $\Fn_{<\kappa}(\kappa,\kappa)$ adds an unbounded function in $\kappa^\kappa$ over the ground model, so $\Fn_{<\kappa}(\lambda\times\kappa,\kappa)$ forces that the $\kappa$-Cohen real at $\xi\in\lambda\menos S$ is not dominated by any $\dot{y}_\alpha$.
\end{proof}

\section{Preservation}\label{SecPres}

In this section, we show a method to preserve $\dfrak_\kappa$ large for $\kappa$ regular. This is a natural generalization of preservation methods by Judah and Shelah \cite{JS} and Brendle \cite{Br}. Our presentation is closer to \cite[Sect. 4]{CM}.

\begin{definition}
   Let $\kappa$ be an infinite cardinal. Say that a poset is \emph{$\kappa^\kappa$-good} if, for any $\Por$-name of a function in $\kappa^\kappa$, there is some $h\in\kappa^\kappa$ (in the ground model) such that, for any $x\in\kappa^\kappa$, if $x\nleq^* h$ then $\Vdash x\nleq^*\dot{y}$.
\end{definition}

\begin{lemma}\label{dombig}
    Any $\kappa^\kappa$-good poset forces that $\dfrak_\kappa\geq|\dfrak_\kappa^V|$.
\end{lemma}
\begin{proof}
   Assume that $\Por$ is a $\kappa^\kappa$-good poset and that $\lambda=\dfrak^V_\kappa$. Let $\gamma<\lambda$ and assume that $\{\dot{y}_\alpha:\alpha<\gamma\}$ is a set of $\Por$-names of functions in $\kappa^\kappa$. For each $\alpha<\gamma$ there is some $h_\alpha\in\kappa^\kappa$ satisfying goodness for $\dot{y}_\alpha$. Since $\gamma<\lambda$, there is some $x\in\kappa^\kappa$ such that $x\nleq^* h_\alpha$ for any $\alpha<\gamma$. Therefore, by goodness, $\Por$ forces that $x\nleq^*\dot{y}_\alpha$.
\end{proof}

The following couple of lemmas illustrate simple examples of $\kappa^\kappa$-good posets.

\begin{lemma}[cf. {\cite[Lemma 1.46]{montoya}}]\label{smallgood}
   If $\kappa$ is regular then any poset of size $\leq\kappa$ is $\kappa^\kappa$-good.
\end{lemma}
\begin{proof}
   Let $\Por$ be a poset of size $\leq\kappa$ and assume that $\dot{y}$ is a $\Por$-name of a function in $\kappa^\kappa$. For each $p\in\Por$ and $\xi<\kappa$ it is clear that there is some $h_p(\xi)<\kappa$ such that $p\nVdash\dot{y}(\xi)\neq h_p(\xi)$. Since $|\Por|\leq\kappa<\bfrak_\kappa$, there is some $h\in\kappa^\kappa$ such that $h_p\leq^* h$ for any $p\in\Por$. It is not hard to see that $x\nleq^* h$ implies $\Vdash x\nleq^*\dot{y}$.
\end{proof}

\begin{lemma}
    If $\kappa$ is regular then any $\kappa$-cc poset is $\kappa^\kappa$-good.
\end{lemma}
\begin{proof}
   Let $\Por$ be a $\kappa$-cc poset and let $\dot{y}$ be a $\Por$-name of a function in $\kappa^\kappa$.

   \begin{claim}
       If $\dot{\alpha}$ is a $\Por$-name of a member of $\kappa$ then there is some $\beta\in\kappa$ such that $\Vdash \dot{\alpha}<\beta$.
   \end{claim}
   \begin{proof}
      Assume the contrary, that is, for any $\beta<\kappa$ there is some $p_\beta\in\Por$ such that $p_\beta\Vdash\beta\leq\dot{\alpha}$. Since $\Por$ is $\kappa$-cc and $\kappa$ is regular, there is some $q\in\Por$ forcing $|\{\beta<\kappa:p_\beta\in\dot{G}\}|=\kappa$, which implies that $q\Vdash\kappa\leq\dot{\alpha}$, a contradiction.
   \end{proof}

   For each $\xi<\kappa$, apply the claim to find some $h(\xi)\in\kappa$ such that $\Vdash \dot{y}(\xi)<h(\xi)$. It is clear that $\Vdash\dot{y}< h$, therefore, $x\nleq^* h$ implies $\Vdash x\nleq\dot{y}$.
\end{proof}

Montoya \cite[Sect. 1.2.2]{montoya} defines a canonical forcing $\Eor_\kappa$ that adds a function in $\kappa^\kappa$ eventually different from the ground model functions in $\kappa^\kappa$, and she proves that $\Eor_\kappa$ is $\kappa^\kappa$-good whenever $\Eor_\kappa$ forces that $\kappa$ is measurable.

We finish this section with the following iteration result.

\begin{lemma}\label{goodit}
    Assume that $\delta$ is a limit ordinal and that $\la\Por_\xi:\xi<\delta\ra$ is a $\lessdot$-increasing sequence of $\kappa^\kappa$-good posets. Let $\Por:=\limdir_{\xi<\delta}\Por_\xi$. If $\cf(\delta)>\kappa$ and $\Por$ is $\cf(\delta)$-cc then $\Por$ is $\kappa^\kappa$-good.
\end{lemma}
\begin{proof}
   If $\dot{y}$ is a $\Por$-name of a function in $\kappa^\kappa$, then there is some $\alpha<\delta$ such that $\dot{y}$ is a $\Por_\alpha$-name, this because $\Por$ is $\cf(\delta)$-cc and $\cf(\delta)>\kappa$. Let $h\in\kappa^\kappa$ be a function obtained from the goodness of $\Por_\alpha$ applied to $\dot{y}$. It is clear that $x\nleq^* h$ implies $\Vdash_\Por x\nleq\dot{y}$.
\end{proof}

\section{Tree forcings and the main results}\label{SecMain}

We review the following notation about trees. Say that $T\subseteq\omega^{<\omega}$ is a \emph{tree} if $\la\ \ra\in T$ and $\forall t\in T\forall s\subseteq t(s\in T)$. Denote by $\Lv_n(T):=T\cap\omega^\omega$ the \emph{$n$-th level of $T$} and, for any $s\in T$, let $T^s:=\{t\in T:s\subseteq t\text{\ or }t\subseteq t\}$, which is also a tree. Denote by $[T]=:\{x\in\omega^\omega:\forall n<\omega(x{\upharpoonright}n\in T)\}$ the set of infinite branches of $T$.

Let $T\subseteq\omega^{<\omega}$ be a tree. Say that $s\in T$ is a \emph{splitting node of $T$} if $s^\frown\la i\ra,s^\frown\la j\ra\in T$ for some $i\neq j$. Denote by $\spl(T)$ the set of splitting nodes of $T$. For $n<\omega$, let $\spl_n(T)$ be the set of $s\in\spl(T)$ such that there are exactly $n$-many splitting nodes strictly below $s$. Given another tree $T'\subseteq\omega^\omega$, write $T'\subseteq_n T$ when $T'\subseteq T$ and there is some $m<\omega$ such that all the elements of $\spl_n(T)$ have length $<m$ and $T'\cap\omega^m=T\cap\omega^m$. Note that $T'\subseteq_{n+1} T$ implies $T'\subseteq_n T$, and that the relation $\subseteq_n$ is transitive.

\begin{definition}\label{Defbtree}
   Let $b:\omega\to\omega\menos\{0\}$. We say that a poset $\Tbb$ is a \emph{$b$-tree forcing notion} if it satisfies the following properties
   \begin{enumerate}[({T}1)]
       \item\label{treeb} $\Tbb$ is a non-empty set of trees contained in $\seq_{<\omega}(b)$.
       \item\label{splb} If $T\in\Tbb$ and $s\in T$, then there is some splitting note $t\in T$ extending $s$.
       \item For $T,T'\in\Tbb$, $T'\leq T$ implies $T'\subseteq T$.
       \item If $T\in\Tbb$ and $s\in T$ then $T^s\in\Tbb$ and $T^s\leq T$.
       \item If $T\in\Tbb$, $n<\omega$ and $\{S_t:t\in\Lv_n(T)\}\subseteq\Tbb$ such that $S_t\leq T^t$ for all $t\in \Lv_n(T)$, then $S:=\bigcup_{t\in \Lv_n(T)} S_t\in \Tbb$, $S\leq T$ and $\{S_t:t\in \Lv_n(T)\}$ is a maximal antichain below $S$.
       \item If $\la T_n:n<\omega\ra$ is a decreasing sequence in $\Tbb$ and $T_{n+1}\subseteq_n T_n$ for al $n<\omega$, then $T:=\bigcap_{n<\omega}T_n\in\Tbb$ and $T\leq T_n$ for all $n<\omega$.
   \end{enumerate}
   When $\Tor$ is a $b$-tree forcing for some $b$ we say that $\Tor$ is a \emph{bounded-tree forcing notion}. Note that (T1) and (T2) imply $b\nleq^*1$.
   Denote by $\Tbb_b$ the poset of all conditions satisfying (T1) and (T2), ordered by $\subseteq$. It is clear that this is a $b$-tree forcing notion.
\end{definition}

\begin{example}
    \begin{enumerate}[(1)]
        \item Recall \emph{Sacks forcing} $\Sor:=\Tbb_2$ (where $2$ represents the constant function with value $2$). It is clearly a $2$-tree forcing notion.
        \item Let $\PTbb_b$ be the poset of conditions $T\in\Tbb_b$ such that, whenever $s\in\spl(T)$, $s{}^\frown\la i\ra\in T$ for every $i\in b(|s|)$. Judah, Goldstern and Shelah \cite{GJS} defined this poset and showed that, under CH, there is a CS iteration of such type of posets forcing $\add(\SNcal)=\aleph_2$ (even more, it shows that $\SNcal=[\R]^{\leq\aleph_1}$). In particular, these tree forcings are used to prove the consistency result (C1) presented in the introduction. 
    \end{enumerate}
\end{example}

\begin{lemma}
    Any $b$-tree forcing notion is proper and strongly $\omega^\omega$-bounding.\footnote{A poset $\Por$ is \emph{strongly $\omega^\omega$-bounding} if for any $p\in\Por$ and any $\Por$-name $\dot{x}$ of a function from $\omega$ into the ground model, there are a function $f$ from $\omega$ into the finite sets and some $q\leq p$ that forces $\dot{x}(n)\in f(n)$ for any $n<\omega$.}
\end{lemma}
\begin{proof}
    The standard argument (see e.g. \cite{GS93}) works thanks to (T6).
\end{proof}

The following result is essential to show that $b$-tree forcings increase $\cov(\SNcal)$. It relies in the notation fixed in Definition \ref{DefIntRel}, and in Lemma \ref{dompart}.

\begin{mainlemma}\label{mainlemma}
    Let $\Tbb$ be a $b$-tree forcing notion and let $D\subseteq\omega^\omega$ be a dominating family of increasing functions. If $\sigma^f\in(\seq_{<\omega}(b))^\omega$ with $\hgt_{\sigma^f}=f^*$ for each $f\in D$ then, for any $T\in\Tor$, there is some $S\leq T$ in $\Tor$ and some $f\in D$ such that $[\sigma^f]_\infty\cap[S]=\emptyset$. In particular,    
    $\Tbb$ forces that $\tau\notin\bigcap_{f\in D}[\sigma^f]_\infty$ where    
    $\tau$ denotes generic real in $\prod b$ added by $\Tbb$. 
\end{mainlemma}
\begin{proof}
    Fix $T\in\Tbb$. Define $f:\omega\to\omega$ such that, for any $t\in \Lv_n(T)$, there is a splitting node of length $<f(n)$ extending $t$. By recursion, define $g(0)=0$ and $g(n+1)=f(g(n))$, which clearly yields an increasing function $g$. Set $I:=(I^g)^{*2}$, that is, $I_n=[g(2n),g(2(n+1))$ for each $n<\omega$. Since $D$ is dominating, by Lemma \ref{dompart} there is some $f\in D$ such that $I\sqsubseteq I^{f^*}$. For $n<\omega$, choose some $k_n$ (if exists) such that $I_{k_n}\subseteq I^{f^*}_n$. Note that there are only finitely many $n<\omega$ for which $k_n$ does not exist.

    Now we define $T_n$ by recursion on $n<\omega$ such that $T_n^t=T^t$ for any $t\in \Lv_{f^*(n)}(T_n)$. Put $T_0=T$. For the successor step, if $k_n$ does not exist then we set $T_{n+1}:=T_n$; else, when $k_n$ exists, for each $t\in\Lv_{g(2k_n)}(T_n)$ choose some $t'\in\Lv_{g(2k_{n}+1)}(T)$ extending $t$ (recall that $f^*(n)\leq g(2k_n)$) such that $t'$ is incompatible with $\sigma^f_{n+1}$. This is possible because there is a splitting node of length $<g(2k_n+1)$ extending $t$ and $|\sigma^f_{n+1}|=f^*(n+1)\geq g(2(k_n+1))$. Put $T_{n+1}:=\bigcup_{t\in\Lv_{f^*(n)}(T_n)}T^{t'}$. For each $t\in\Lv_{f^*(n)}(T_n)$, $T_{n+1}$ contains a splitting node of length $<f^*(n+1)$ extending $t'$. This indicates that $\la T_n:n<\omega\ra$ satisfies the conditions of (T6), so $S:=\bigcap_{n<\omega}T_n \in\Tbb$ and $S\leq T$. Even more, any branch of $S$ is incompatible with $\sigma^f(k)$ for all but finitely many $k<\omega$, so $[\sigma^f]_\infty\cap[S]=\emptyset$.
\end{proof}

\begin{corollary}\label{marc}
    $\SNcal\subseteq s^0$.
\end{corollary}
\begin{proof}
   Apply Main Lemma \ref{mainlemma} to $\Tor=\Sor$.
\end{proof}

Now we are ready to prove the main results of this paper.

\begin{theorem}\label{SacksSN}
   Assume CH. Then, any CS (countable support) iteration of length $\omega_2$ of bounded-tree forcing notions forces $\cov(\SNcal)=\aleph_2$.
\end{theorem}
\begin{proof}
   Assume that $\la\Por_\alpha:\alpha\leq\omega_2\ra$ results from such iteration and fix any dominating family $D$ of increasing functions in the ground model (by CH, $|D|=\aleph_1$). Let $D^*:=\{f^* : f\in D\}$, which is also a dominating family. Assume that $\{\dot{X}_\xi:\xi<\omega_1\}$ is a family of $\Por$-names of members of $\SNcal(2^\omega)$. For each $\xi<\omega_1$ and $f\in D$, there is a $\Por$-name $\dot{\sigma}^f_\xi$ for a function in $(2^{<\omega})^\omega$ such that $\Por$ forces $\hgt_{\dot{\sigma}^f_\xi}=f^*$ and $X_\xi\subseteq[\dot{\sigma}^f_\xi]_\infty$. Since $\Por_{\omega_2}$ has $\aleph_2$-cc,  there is some $\alpha<\aleph_1$ such that $\dot{\sigma}^f_\xi$ is a $\Por_\alpha$-name for each $f\in D$ and $\xi<\omega_1$. Let $\dot{\Tor}$ be a $\Por_\alpha$-name of a bounded-tree forcing notion such that $\Por_{\alpha+1}=\Por_\alpha\ast\dot{\Tor}$.

   Fix a $\Por_\alpha$-generic set $G$ over $V$. Work in $V[G]$. Let $b:\omega\to\omega$ be a function such that $\Tor:=\dot{\Tor}[G]$ is a $b$-tree forcing notion. Thanks to the maps $F_2$ and $F_b$ (see Section \ref{SecPre}), since $D^*$ is still a dominating function in $V[G]$ and $\bigcap_{f\in D}[\sigma^f_\xi]_\infty\in\SNcal(2^\omega)$ for each $\xi<\omega_1$, we can find some $\rho^f_\xi\in(\seq_{<\omega}(b))^\omega$ with $\hgt_{\rho^f_\xi}=f^*$ for each $f\in D$ and $\xi<\omega_1$ such that $F_b^{-1}F_2''\bigcap_{f\in D}[\sigma^f_\xi]_\infty\subseteq[\rho^f_\xi]_\infty$. By Main Lemma \ref{mainlemma}, $\Tor$ forces $\tau_\alpha\notin\bigcap_{f\in D}[\rho^f_\xi]_\infty$ for each $\xi<\omega_1$ (here, $\tau_\alpha\in\prod b$ is the generic real added by $\Tor$), so $F^{-1}_2(F_b(\tau_\alpha))\notin\bigcap_{f\in D}[\sigma^f_\xi]_\infty$ (since $\tau_\alpha$ is a generic real, it can be shown by a density argument that $F_b(\tau_\alpha)$ has a unique pre-image under $F_2$).

   Therefore, $\Por_{\omega_2}$ forces that $F^{-1}_2(F_b(\tau_\alpha))\notin\bigcup_{\xi<\omega_1} X_\xi$.
\end{proof}

\begin{theorem}\label{3SN}
   Assume CH and that $\lambda$ is an infinite cardinal such that $\lambda^{\aleph_1}=\lambda$. Then, there is a proper $\omega^\omega$-bounding poset with $\aleph_2$-cc forcing $\cof(\Nwf)=\afrak=\ufrak=\ifrak=\aleph_1$, $\cov(\SNcal)=\aleph_2$ and $\cof(\SNcal)=\lambda$. In particular, it is consistent with ZFC that $\non(\SNcal)<\cov(\SNcal)<\cof(\SNcal)$.
\end{theorem}
\begin{proof}
   We show that $\Fn_{<\omega_1}(\lambda\times\omega_1,\omega_1)$ followed by the CS iteration of $\Sor$ of length $\aleph_2$ is the desired poset. By CH, $\Fn_{<\omega_1}(\lambda\times\omega_1,\omega_1)$ has $\aleph_2$-cc, and it is clear that it is $<\omega_1$-closed, so it is proper and preserves cofinalities (and it is obviously $\omega^\omega$-bounding since it does not add new reals). Even more, in the $\Fn_{<\omega_1}(\lambda\times\omega_1,\omega_1)$-forcing extension, CH still holds, $2^{\aleph_1}=\lambda$ and, by Lemma \ref{kappaCohen}, $\dfrak_{\omega_1}=\lambda$.

   Now work in the  $\Fn_{<\omega_1}(\lambda\times\omega_1,\omega_1)$-extension. Let $\Qor=\la\Por_\alpha,\Sor:\alpha<\omega_2\ra$ be the CS iteration of Sacks forcing of length $\omega_2$. It is clear that $\Qor$ forces $\cof(\Nwf)=\afrak=\ufrak=\ifrak=\aleph_1$ and, by Theorem \ref{SacksSN}, it forces $\cov(\SNcal)=\cfrak=\aleph_2$. In addition, since $\supcof\leq\cof(\SNcal)$, by Theorem \ref{Ycof}, $\Qor$ forces that $\cof(\SNcal)=\dfrak_{\omega_1}$.

   It remains to show that $\Qor$ forces $\dfrak_{\omega_1}=\lambda$. Since $\Qor$ has $\aleph_2$-cc and size $\aleph_2$, it forces $2^{\aleph_1}=\lambda$. On the other hand, for each $\alpha<\omega_2$, $|\Por_\alpha|=\aleph_1$, so $\Por_\alpha$ is $\omega_1^{\omega_1}$-good by Lemma \ref{smallgood}. Hence, by Lemma \ref{goodit}, $\Qor$ is $\omega_1^{\omega_1}$-good and, by Lemma \ref{dombig}, $\Qor$ forces $\lambda\leq\dfrak_{\omega_1}$.
\end{proof}

\begin{remark}
   In the proof above it can be shown in addition that the first $\omega_2$-many $\omega_1$-Cohen reals form an unbounded family of $\omega_1^{\omega_1}$ even after the iteration of Sacks forcing. Hence, the final model satisfies $\bfrak_{\omega_1}=\aleph_2$.
\end{remark}

\begin{remark}
   Judah, Miller and Shelah \cite{JMS} have proved that, in Sacks model, $\add(s^0)=\aleph_1$ and $\cov(s^0)=\cfrak$. So Corollary \ref{marc} also implies that $\cov(\SNcal)=\cfrak$ in this model.
\end{remark}

\section{Discussions}\label{SecDisc}

Main Lemma \ref{mainlemma} can also be proved for Silver-like type of posets, or more generally, for lim-sup creature type forcing notions obtained by finitary creating pairs as in \cite{RS}. Therefore, these type of posets can be included as iterands in Theorem \ref{SacksSN}. Moreover, it can be concluded that $\SNcal$ is contained in the Marczewski-type ideal corresponding to Silver forcing.

Bartoszy\'nski and Shelah \cite[Thm. 3.3]{BaS} proved that $\non(\SNcal)$ can be increased by CS products of Silver-like posets. In fact, the same argument applies to CS products of posets of the form $\PTbb_b$ with $b$ diverging to infinity. Concretely, assuming CH, if $\kappa^{\aleph_0}=\kappa$, $I$ is a set of size $\kappa$ and $\{ b_i:i\in I\}\subseteq\omega^\omega$ is a family of functions diverging to infinity, then the CS product of $\PTbb_{b_i}$ with $i\in I$ forces $\dfrak=\aleph_1$ (it is $\omega^\omega$-bounding) and $\non(\SNcal)=\cfrak=\kappa$.

A very natural question that comes from our main result is whether a version of Theorem \ref{SacksSN} for CS products can be proved. By methods like in \cite{GS93,KM} it can be shown that any CS product of bounded-tree forcing notions remains proper and strongly $\omega^\omega$-bounding. However, it is not obvious how the proof of Main Lemma \ref{mainlemma} can be translated to show that such a CS product increases $\cov(\SNcal)$. This would generalize the consistency result of Theorem \ref{SacksSN} in the sense that $\cov(\SNcal)$ could be forced larger than $\aleph_2$.

By well known methods and results from \cite{Yorioka}, the following open problem is the only one remaining to settle that the diagram of inequalities in Figure \ref{FigSN} is complete.

\begin{figure}
\begin{center}
  \includegraphics{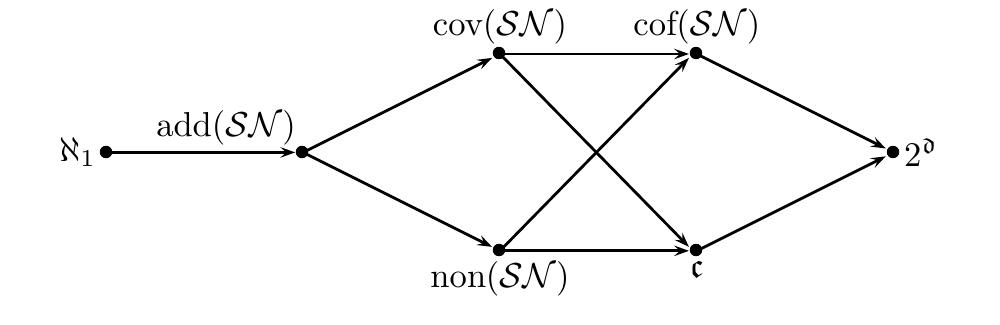}
  \caption{The arrows mean that $\leq$ is provable in ZFC.}
  \label{FigSN}
\end{center}
\end{figure}

\begin{question}\label{Qadd}
   Is it consistent with ZFC that $\add(\SNcal)<\min\{\cov(\SNcal),\non(\SNcal)\}$?
\end{question}

This work provides the first example where 3 cardinals associated with $\SNcal$ can be pairwise different. To go one step further, we propose the following problem.

\begin{question}\label{Q4}
Is it consistent with ZFC that the four cardinal invariants associated with $\SNcal$ are pairwise different?
\end{question}

The following idea may be useful to answer the question above.
Quite recently, the first and third authors with Brendle \cite{BCM} constructed a ccc poset forcing
\[\add(\Nwf)=\add(\Mwf)<\cov(\Nwf)=\non(\Mwf)<\cov(\Mwf)=\non(\Nwf)<\cof(\Mwf)=\cof(\Nwf).\]
In the same model, $\cov(\SNcal)=\cov(\Nwf)<\non(\SNcal)=\non(\Nwf)$ by (S3) and because this model is obtained by a FS iteration of length with cofinality $\mu$ (where $\mu$ is the desired value for $\non(\Mwf)$), and it is well known that such cofinality becomes an upper bound of $\cov(\SNcal)$ (see e.g. \cite[Lemma 8.2.6]{BJ}). However, tools to deal with $\add(\SNcal)$ and $\cof(\SNcal)$ in this situation are still unknown.

It would also be very useful to have a stronger characterization of $\cof(\SNcal)$. So far, Theorem \ref{Ycof} is restricted to $\minadd=\supcof$, which implies $\add(\SNcal)=\non(\SNcal)$, so another characterization that allows $\add(\SNcal)<\non(\SNcal)$ would lead to methods to solve Question \ref{Q4}.


{\small
\bibliography{left}
\bibliographystyle{alpha}
}

\end{document}